\documentclass{amsart}
\usepackage{graphics,amssymb}

\newcommand{\C}{\mathbb{C}}
\newcommand{\D}{\mathbb{D}}

\newcommand{\R}{\mathbb{R}}

\newcommand{\Harm}{\mbox{BHarm}}

\newtheorem{theorem}{Theorem}[section]
\newtheorem{lemma}[theorem]{Lemma}

\theoremstyle{definition}

\theoremstyle{theorem}

\theoremstyle{theorem}
\newtheorem{proposition}[theorem]{Proposition}

\theoremstyle{theorem}

\theoremstyle{theorem}

\theoremstyle{definition}

\theoremstyle{theorem}

\numberwithin{equation}{section}

\begin{document}
\title {An operator-theoretic existence proof of solutions to planar Dirichl\'et problems}
\author{Timothy H. McNicholl}
\address{Department of Mathematics\\
              Lamar University\\
              Beaumont, Texas 77710 USA}
\email{timothy.h.mcnicholl@gmail.com}

\begin{abstract}
By using some elementary techniques from operator theory, we prove constructively prove the 
existence of solutions to Dirichl\'et problems for planar Jordan domains with at least two boundary curves.
An iterative method is thus obtained, and explicit bounds on the error in the resulting approximations are given.  Finally, a closed form for the solution is given.  No amount of differentiability of the boundary is assumed.
\end{abstract}
\subjclass[2010]{31, 47}

\maketitle

\section{Introduction}

Suppose we are given a bounded Jordan domain $D \subseteq \C$ and a piecewise continuous function $f : \partial D \rightarrow \R$.  The resulting 
\emph{Dirichl\'et problem} is to find a harmonic function on $D$, $u$, such that 
\begin{eqnarray}
\lim_{z \rightarrow \zeta} u(z) & = & f(\zeta)\label{eqn:DIRI}
\end{eqnarray}
for all $\zeta \in \partial D$ at which $f$ is continuous.  It is well-known that this Dirichl\'et problem has a solution.  Once the existence of $u$ is demonstrated, uniqueness follows immediately from the Maximum Principle for harmonic functions.  That is, there is exactly one harmonic function on $D$, $u$, for which Equation (\ref{eqn:DIRI}) holds. See, \emph{e.g.}, Chapters I and II of \cite{Garnett.Marshall.2005}.  Accordingly, we denote this function $u_f$.  Dirichl\'et problems for other kinds of domains exist but will not be considered here.

 When $D$ is simply connected, that is when $D$ is the interior of a Jordan curve, it is fairly straightforward to prove the existence of $u_f$.  
 In particular, when $D$ is the unit disk $\D$, 
 \[
 u_f(z) = \frac{1}{2\pi} \int_{\partial \D} f(\zeta) P(z, \zeta) ds_\zeta\mbox{,}
 \]
 where 
 \[
 P(z, \zeta) = \frac{1 - |z|^2}{|z - \zeta|^2}
 \]
 and $ds_\zeta$ is the differential of arc length with respect to the variable $\zeta$.
 The function $P$ is of course called the \emph{Poisson kernel}.  Since the composition of a harmonic function with an analytic function yields a harmonic function, the existence of $u_f$ for simply connected Jordan domains now follows from the Riemann Mapping Theorem and the Carath\'eodory Theorem.

When $D$ is bounded by more than one Jordan curve, there are at least two methods available to prove the existence of $u_f$.
One is an extreme-value argument as in Section 6.4.2 of \cite{Ahlfors.1978}.  If one desires a constructive proof, a natural choice is the Schwarz alternating method which is described in Section IV.2 of \cite{Courant.Hilbert.1989.2} and in Chapter II of \cite{Garnett.Marshall.2005}.  A sequence of harmonic functions $u_1 \geq u_2 \geq u_3 \ldots$ that converges to $u$ is produced thereby.  However, the proof does not give any information about the rate of convergence which is essential for error control when designing a numerical method.  Nevertheless, it is hinted in the exercises of \cite{Garnett.Marshall.2005} that the Schwarz alternating method can be turned into an integral operator which in turn leads to a numerical method with error control.

Here, we will turn this chain of ideas on its head and give what appears to be a new and constructive proof of the existence of solutions to Dirichl\'et problems for planar Jordan domains based on some fairly elementary ideas from operator theory.  Namely, we will first define an integral operator and then show it has a fixed point $u$.  We will then show that this fixed point extends to a harmonic function on $D$ for which Equation (\ref{eqn:DIRI}) holds.  We will also give explicit
bounds on the error in the resulting sequence of iterations.  Finally, we use these results to derive a closed form for the solution to the Dirichl\'et problem.  All results hold for an arbitrary Jordan domain with at least two boundary curves, even those whose boundary is nowhere differentiable.

\section{A few preliminaries}

When $r$ is a positive real and $z_0 \in \C$, let $D_r(z_0)$ denote the open disk whose center is $z_0$ and whose radius is $r$.  Let $\D = D_1(0)$.  

Let $\Harm(D)$ denote the space of bounded harmonic functions on $D$ with the sup norm.

\begin{proposition}\label{prop:COMPLETE}
If $D$ is open, then $\Harm(D)$ is complete.
\end{proposition}

\begin{proof}
Let $u_1, u_2, \ldots$ be a Cauchy sequence in $\Harm(D)$.  
It follows that this sequence converges uniformly to a function $u : D \rightarrow \R$.
It follows that $u$ is continuous and bounded.  It remains to show that $u$ is harmonic.  
To do so, we use the mean value property.  Accordingly, suppose $\overline{D_r(z_0)} \subseteq D$.  Then, 
\begin{eqnarray*}
u(z_0) & = & \lim_{n \rightarrow \infty} u_n(z_0)\\
& = & \lim_{n \rightarrow \infty} \frac{1}{2\pi} \int_0^{2\pi} u_n(z_0 + re^{i \theta}) d\theta\\
& = & \frac{1}{2\pi} \int_0^{2\pi} \lim_{n \rightarrow \infty} u_n(z_0 + re^{i \theta}) d \theta\\
& = & \frac{1}{2\pi} \int_0^{2\pi} u(z_0 + re^{i\theta}) d\theta.
\end{eqnarray*} 
It follows that $u$ is harmonic.
\end{proof}

\section{An operator-theoretic existence proof}\label{sec:OPERATOR}

Suppose $D$ is a bounded Jordan domain.  Suppose $f : \partial D \rightarrow \R$ is piecewise continuous.

Let $\sigma_1, \ldots, \sigma_{n-1}$, $\tau_1, \ldots, \tau_{n-1}$ be pairwise disjoint arcs such that 
$D_1 =_{df} D - \bigcup \tau_j$ and $D_2 =_{df} D - \bigcup \sigma_j$ are simply connected.  The case when $D$ is bounded by four curves is illustrated in Figure \ref{fig:ARCS}.
\begin{figure}[!h]
\resizebox{4.5in}{4.5in}{\includegraphics{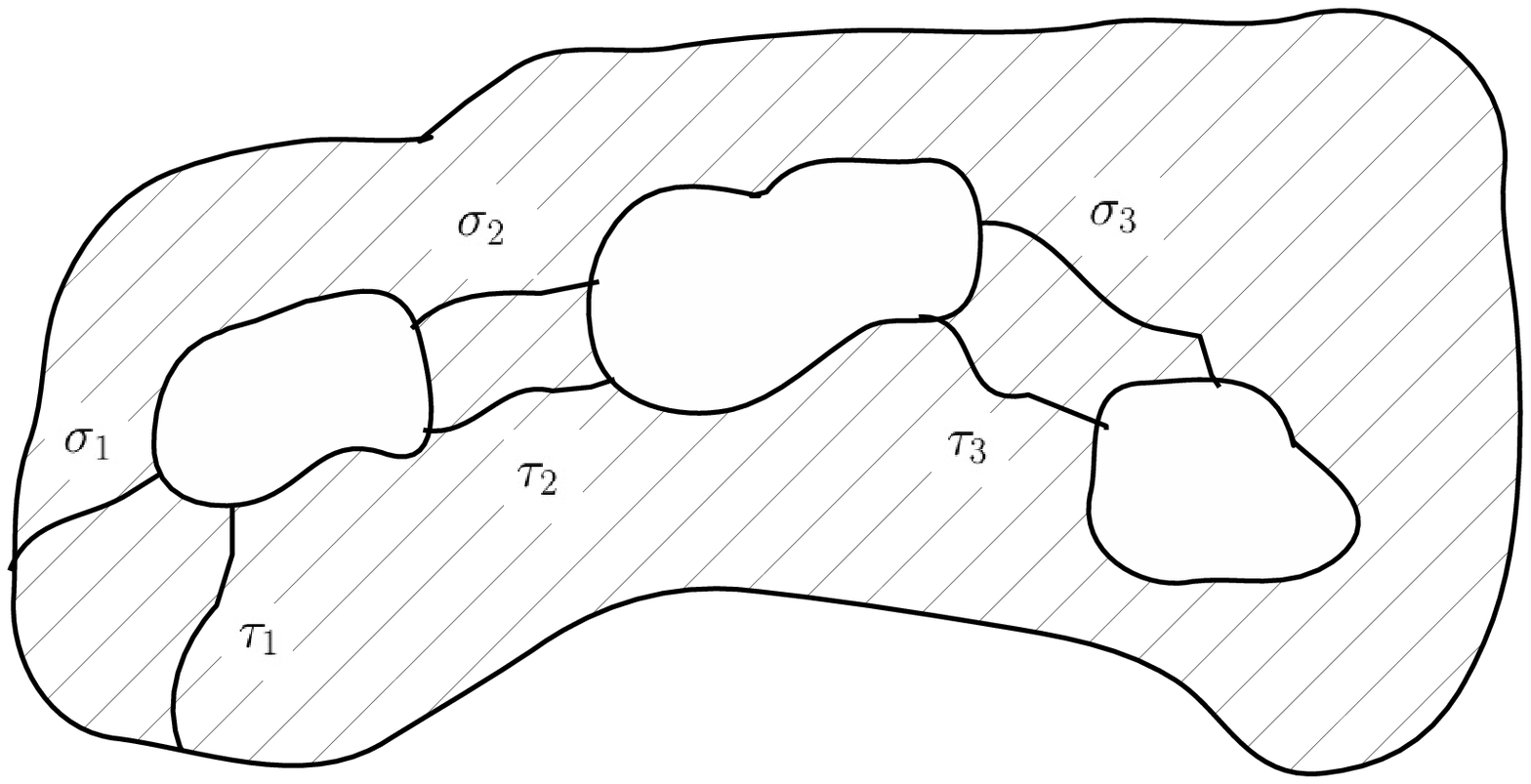}}
\caption{ }\label{fig:ARCS}
\end{figure}
Let $\sigma = \bigcup_j \sigma_j$, and let $\tau = \bigcup_j \tau_j$.

Let $\phi_j$ be a continuous map of $\overline{\D}$ onto $\overline{D_j}$ that is conformal on $\D$.  Thus, each point of $\overline{D} - \sigma$ has exactly one preimage under $\phi_1$, and each point of $\overline{D} - \tau$ has exactly one preimage under $\phi_2$.  The existence of these maps follows from Theorem 2.1 of \cite{Pommerenke.1992}.  
Let: 
\begin{eqnarray*}
A_j & = & \phi_j^{-1}[\partial D]\\
B_1 &= & \phi_1^{-1}[\sigma]\\
B_2 & = & \phi_2^{-1}[\tau]
\end{eqnarray*}

Let $H_1$ be the harmonic function on $D_1$ determined by the boundary conditions 
\[
H_1(\zeta) = \left\{\begin{array}{cc}
						f(\zeta) & \mbox{if $\zeta \in \partial D$}\\
						0 & \mbox{if $\zeta \in \sigma - \partial D$}\\
						\end{array}
						\right. ,
\]
and let $H_2$ be the harmonic function on $D_2$ determined by the boundary conditions
\[
H_2(\zeta) = \left\{\begin{array}{cc}
						f(\zeta) & \mbox{if $\zeta \in \partial D$}\\
						0 & \mbox{if $\zeta \in \tau - \partial D$}\\
						\end{array}
						\right.
\]
It follows that $H_1$ and $H_2$ have the integral forms:
\begin{eqnarray*}
H_1(z) & = &  \frac{1}{2\pi} \int_{A_1} P(\phi_1^{-1}(z), \zeta) f(\phi_1(\zeta)) ds_\zeta\\
H_2(z) & = & \frac{1}{2\pi} \int_{A_2} P(\phi_2^{-1}(z), \zeta) f(\phi_2(\zeta)) ds_\zeta\\
\end{eqnarray*}
Let $H_3$ be the harmonic function on $D_1$ determined by the boundary conditions
\[
H_3(\zeta) = \left\{\begin{array}{cc}
					H_2(\zeta) & \mbox{if $\zeta \in \sigma$}\\
					0 & \mbox{otherwise}\\
					\end{array}
					\right.
\]
Thus, $H_3$ has the integral form
\begin{eqnarray}
H_3(z) & = & \frac{1}{2\pi} \int_{B_1} P(\phi_1^{-1}(z), \zeta) H_2(\phi_1(\zeta)) ds_\zeta.
\end{eqnarray}
Finally, let 
\begin{eqnarray}
h & = & H_1 + H_3. 
\end{eqnarray}
For each $\zeta_1 \in B_2$, let $K(\cdot, \zeta_1)$ be the harmonic function on $D_1$ defined by the boundary conditions
\[
K(\zeta, \zeta_1) = \left\{ \begin{array}{cc}
								0 &  \zeta \not \in \sigma\\
								2\pi P(\phi_2^{-1}(\zeta), \zeta_1) & \zeta \in \sigma\\
								\end{array}
								\right.
\]
Thus, $K$ has the integral form
\[
K(z, \zeta_1) = \int_{B_1} P(\phi_2^{-1}\phi_1(\zeta), \zeta_1) P(\phi_1^{-1}(z), \zeta) ds_\zeta.
\]

Note that $\phi_2^{-1}(\zeta)$ is bounded away from $B_2$ as $\zeta$ ranges over $\sigma$.  Accordingly, let 
\begin{eqnarray*}
m & = & \max\left\{ \int_{B_2} P(\phi_2^{-1}(\zeta), \zeta_1) ds_{\zeta_1}\ :\ \zeta \in \sigma\right\}.
\end{eqnarray*} 
It follows that $m < 2\pi$.

We now define an operator on $\Harm(D_1)$.  When $v$ is harmonic on $D_1$, let $F(v)$ denote the function on $D_1$ defined by the equation
\[
F(v)(z) = h(z) + \frac{1}{(2\pi)^2} \int_{B_2} K(z, \zeta_1) v(\phi_2(\zeta_1)) ds_{\zeta_1}.
\]

The key lemma is the following.

\begin{lemma}\label{lm:CONTRACTION}
$F$ is a contraction map on $\Harm(D_1)$.  In particular, for all $v_1, v_2 \in \Harm(D_1)$, 
\[
\parallel F(v_1) - F(v_2) \parallel_\infty \leq \frac{m}{2\pi} \parallel v_1 - v_2 \parallel_\infty.
\]
\end{lemma}

\begin{proof}
We first show that $F(v)$ is a harmonic function on $D_1$ whenever $v$ is.  That is, 
$F$ maps $\Harm(D_1)$ into $\Harm(D_1)$.
This can be seen by expanding $F(v)(z)$ and applying Fubini's Theorem so as to obtain
\begin{eqnarray}
\nonumber F(v)(z) & = & H_1(z) \\
\nonumber & + & \frac{1}{2\pi} \int_{B_1} \left[ \frac{1}{2\pi} \int_{A_2} f(\phi_2(\zeta_1)) P(\phi_2^{-1}\phi_1(\zeta), \zeta_1) ds_{\zeta_1} \right] P(\phi_1^{-1}(z), \zeta) ds_\zeta\\
& + & \frac{1}{2\pi} \int_{B_1} \left[ \frac{1}{2\pi} \int_{B_2} v(\phi_2(\zeta_1)) \nonumber P(\phi_2^{-1}\phi_1(\zeta), \zeta_1) ds_{\zeta_1}\right] P(\phi_1^{-1}(z), \zeta) ds_\zeta\\\label{eqn:EXPANDED}
\end{eqnarray}
Each summand in Equation (\ref{eqn:EXPANDED}) defines a harmonic function on $D_1$.  Thus, $F(v)$ is harmonic on $D_1$ whenever $v$ is.  

We now show that $F$ is a contraction map.
It follows from Fubini's Theorem that 
\[
\int_{B_2} K(z, \zeta_1) ds_{\zeta_1} = \int_{B_1} \left[ \int_{B_2} P(\phi_2^{-1}\phi_1(\zeta), \zeta_1) ds_{\zeta_1} \right] P(\phi_1^{-1}(z), \zeta) ds_\zeta.
\]
However, 
\begin{eqnarray*}
 \int_{B_1} \left[ \int_{B_2} P(\phi_2^{-1}\phi_1(\zeta), \zeta_1) ds_{\zeta_1} \right] P(\phi_1^{-1}(z), \zeta) ds_\zeta & \leq & m \int_{B_1} P(\phi_1^{-1}(z), \zeta) ds_{\zeta}\\
 & \leq & 2\pi m.
\end{eqnarray*}
The conclusion follows.
\end{proof}

So, let $u$ be the fixed point of $F$.

\begin{lemma}\label{lm:SUCCESS}
$u$ extends to a harmonic function on $D$ such that $f(\zeta) = \lim_{z \rightarrow \zeta} u(z)$ for all $\zeta \in \partial D$ at which $f$ is continuous.
\end{lemma}

\begin{proof}
By Lemma \ref{lm:CONTRACTION}, $u$ is harmonic on $D_1$.  It follows from Equation (\ref{eqn:EXPANDED}) and Fubini's Theorem that 
\[
u(z) = F(u)(z) = H_1(z) + \frac{1}{2\pi} \int_{B_1} g(\phi_1(\zeta)) P(\phi_1^{-1}(z), \zeta) ds_\zeta
\]
where
\[
g(z) = \frac{1}{2\pi} \int_{A_2} P(\phi_2^{-1}(z), \zeta_1) f(\phi_2(\zeta_1)) ds_{\zeta_1} + \frac{1}{2\pi} \int_{B_2} P(\phi_2^{-1}(z), \zeta_1) u(\phi_2(\zeta_1)) ds_{\zeta_1}.
\]
Thus, $g$ is the harmonic function on $D_2$ defined by the boundary conditions
\[
g(\zeta) = \left\{\begin{array}{cc}
					f(\zeta) & \zeta \in \partial D\\
					u(\zeta) & \zeta \in \tau - \partial D\\
					\end{array}
					\right.
\]
It also follows that $u$ is the harmonic function on $D_1$ defined by the boundary conditions
\[
u(\zeta) = \left\{ \begin{array}{cc}
						f(\zeta) & \zeta \in \partial D\\
						g(\zeta) & \zeta \in \sigma - \partial D\\
						\end{array}
						\right.
\]
Hence, $g(\zeta) = u(\zeta)$ for all $\zeta \in \sigma$.  Decompose $D$ into the simply connected domains $S_1, \ldots, S_n$ as in Figure \ref{fig:DECOMP}.  
It follows that $u(\zeta) = g(\zeta)$ for all $\zeta \in \partial S_j$.  It then follows that $u(z) = g(z)$ for all $z \in S_j$.  It then follows that $u$ extends to a harmonic function on $D$ that solves the given Dirichl\'et problem.
\begin{figure}[!h]
\resizebox{4.5in}{4.5in}{\includegraphics{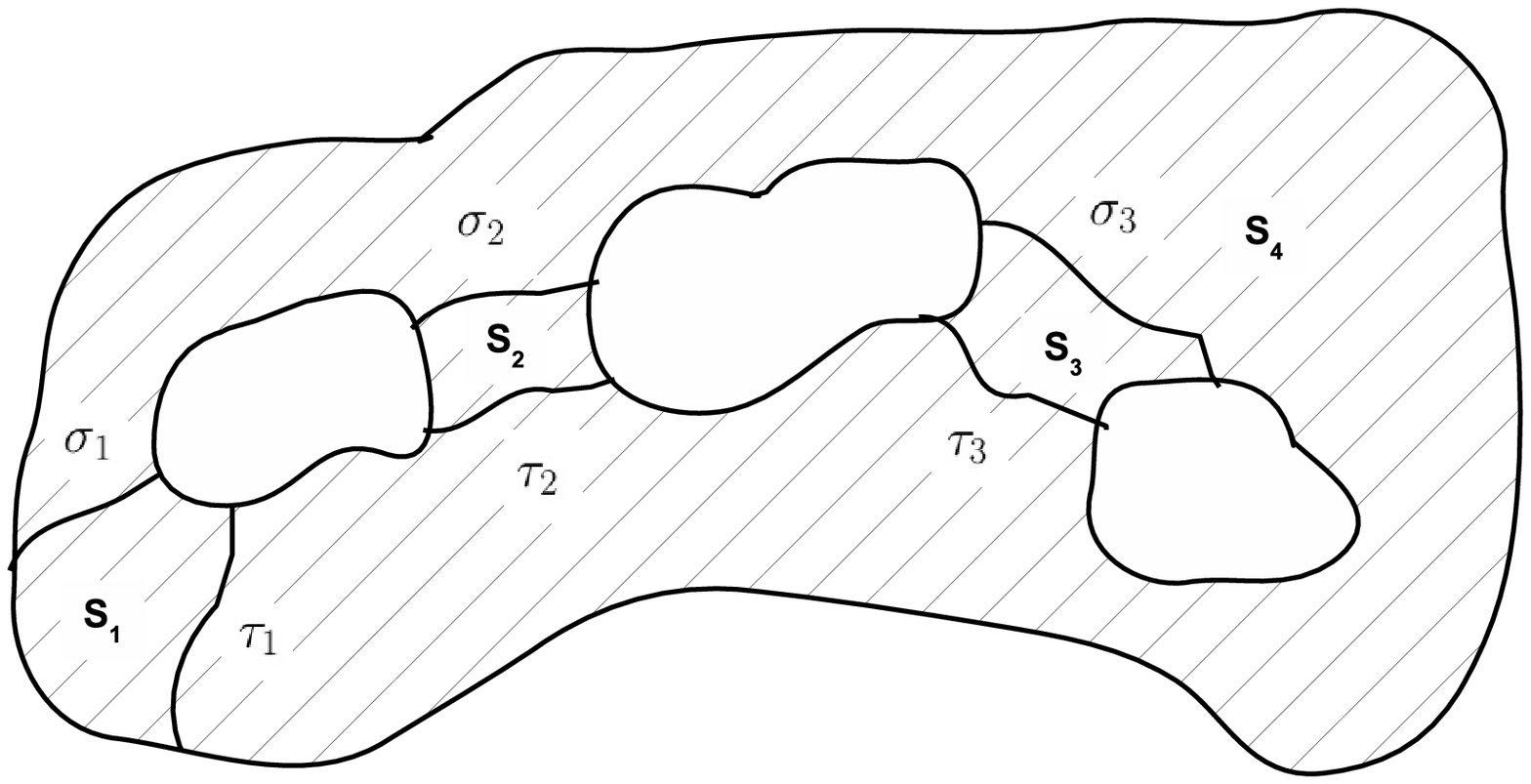}}
\caption{}\label{fig:DECOMP}
\end{figure}
\end{proof}

A few moments of reflection will reveal that what the operator $F$ does is the basic step in the Schwarz alternating method.  However, by couching things in the framework of operator theory, it becomes possible to prove convergence using the contraction mapping theorem.  The proof thus gives more information than would be obtained by using Harnack's Principle since it reveals the rate of convergence of the construction which is important for computation.

\section{A closed form}\label{sec:CLOSED}

\newcommand{\K}{\mbox{\bf K}}
We first perform a procedure similar to kernel iteration.  
Let
\begin{eqnarray*}
\K^{(1)}(z, \zeta_1) & = & K(z, \zeta_1)\\
\K^{(n+1)}(z, \zeta_1) & = & \int_{B_2^n} K(z, \zeta_{n+1}) \prod_{j=1}^n K(\phi_2(\zeta_{j+1}), \zeta_j) ds_{\zeta_2} \ldots ds_{\zeta_{n+1}}\ \ \mbox{$n = 1, 2, \ldots$}\\
\K(z, \zeta_1) & = & \sum_{n=1}^\infty \frac{1}{(2\pi)^{2n}}\K^{(n)}(z, \zeta_1).
\end{eqnarray*}

\begin{theorem}\label{thm:CLOSED.FORM.1}
Suppose $D$ is a planar Jordan domain with at least two boundary curves and that $f : \partial D \rightarrow \R$ is piecewise continuous.  Let $h$, $B_j$, \emph{etc.} be as in Section \ref{sec:OPERATOR}.  Then, 
\begin{eqnarray}
u_f(z) & = & h(z) + \int_{B_2} \K(z, \zeta_1) h(\phi_2(\zeta_1)) ds_{\zeta_1}\label{eqn:CLOSED}
\end{eqnarray}
Furthermore, 
\begin{eqnarray}
\nonumber\left| h(z) + \sum_{n=1}^t \frac{1}{(2\pi)^{2n}} \int_{B_2} \K^{(n)}(z, \zeta_1) h(\phi_2(\zeta_1)) d_{s_{\zeta_1}} - u(z) \right| & \leq & \left( \frac{m}{2\pi} \right)^{t+1} \frac{2\pi}{2\pi - m} \parallel f \parallel_\infty. \\
\label{eqn:RATE} 
\end{eqnarray}
\end{theorem}

\begin{proof}
\newcommand{\zero}{\mathbf 0}
Let $\zero$ denote the zero function on $D_1$.  Again, let $u$ denote the fixed point of $F$.  It follows that $u$ is the restriction of $u_f$ to $D_1$.  Let $z \in D_1$.  By induction on $t = 1, 2, 3, \ldots$, 
\[
F^{t+1}(\zero)(z) = h(z) + \sum_{n=1}^t \frac{1}{(2\pi)^{2n}} \int_{B_2} \K^{(n)}(z, \zeta_1) h(\phi_2(\zeta_1)) ds_{\zeta_1}.
\]
It follows from Lemma \ref{lm:CONTRACTION} that 
\begin{eqnarray*}
\parallel F^j(\zero) - F^{j+1}(\zero) \parallel_\infty & \leq & \left(\frac{m}{2\pi} \right)^j \parallel F(\zero) \parallel_\infty\\
& = & \left( \frac{m}{2\pi} \right)^j \parallel h \parallel_\infty.
\end{eqnarray*}
By the Maximum Principle, $\parallel h \parallel_\infty \leq \parallel f \parallel_\infty$.  By a fairly standard calculation, 
\begin{eqnarray*}
\parallel F^{t+1}(\zero) - u \parallel_\infty & \leq & \sum_{n = t+1}^\infty \left( \frac{m}{2\pi}\right)^k \parallel f \parallel_\infty\\
& = & \left( \frac{m}{2\pi} \right)^{t+1} \frac{2\pi}{2\pi - m} \parallel f \parallel_\infty. 
\end{eqnarray*}
It follows that Equations (\ref{eqn:CLOSED}) and (\ref{eqn:RATE}) hold whenever $z \in D_1$.  Hence, by continuity, they hold for all $z \in D$.
\end{proof}

We mention here that D. Crowdy and J. Marshall have obtained closed formulas for Green's 
function in multiply connected domains.  See \cite{Crowdy.Marshall.2007}.

\bibliographystyle{amsplain}
\def\cprime{$'$}
\providecommand{\bysame}{\leavevmode\hbox to3em{\hrulefill}\thinspace}
\providecommand{\MR}{\relax\ifhmode\unskip\space\fi MR }
\providecommand{\MRhref}[2]{%
  \href{http://www.ams.org/mathscinet-getitem?mr=#1}{#2}
}
\providecommand{\href}[2]{#2}

\end{document}